\documentclass[a4paper,12pt]{amsart}

\usepackage[english]{babel}
\usepackage{amsmath,amssymb}

\usepackage[latin1]{inputenc}

\usepackage[usenames]{color}
\usepackage{graphicx}
\usepackage{psfrag}

\theoremstyle{plain}
\newtheorem{definition}{Definition}[section]
\newtheorem{proposition}[definition]{Proposition}
\newtheorem{theorem}[definition]{Theorem}

\newtheorem{lemma}[definition]{Lemma}

\theoremstyle{definition}

\newtheorem{example}[definition]{Example}

\newcommand{\Z}{\mathbb{Z}}

\newcommand{\G}{{\mathcal{G}}}

\newcommand{\link}{\operatorname{link}}
\renewcommand{\star}{\operatorname{star}}
\newcommand{\astar}{\operatorname{astar}}

\title[On the connectivity of manifold graphs]{On the connectivity of \\ manifold graphs}
\author{Anders Bj\"orner}
\address{Royal Institute of Technology, Department of Mathematics,
S-100 44 Stockholm, Sweden}
\email{bjorner@math.kth.se}
\author{Kathrin Vorwerk}
\address{Royal Institute of Technology, Department of Mathematics,
S-100 44 Stockholm, Sweden}
\email{vorwerk@math.kth.se}
\thanks{{Research supported by the Knut and Alice Wallenberg Foundation,
grant KAW.2005.0098}}

\begin{document}

\begin{abstract}
This paper is concerned with lower bounds for the connectivity of graphs
(one-dimensional skeleta) of triangulations of compact manifolds.
We introduce a structural invariant $b_{\Delta}$ of a simplicial $d$-manifold $\Delta$
taking values in the range $0\le b_{\Delta} \le d-1$. 
The main result is that $b_\Delta$ influences connectivity  in the following way:
\emph{
The graph of a $d$-dimensional
simplicial compact manifold $\Delta$ is $(2d-b_{\Delta})$-connected.} 

The parameter $b_{\Delta}$
has the property that  $b_{\Delta} =0$ if the complex $\Delta$ is flag.
Hence, our result interpolates between Barnette's theorem (1982) that all $d$-manifold graphs are $(d+1)$-connected 
 and Athanasiadis' theorem (2011) that flag $d$-manifold graphs are $2d$-connected.
 
The definition of $b_{\Delta}$ involves the
concept of {\em banner} triangulations of manifolds, a generalization of flag triangulations.
\end{abstract}

\maketitle

\section{Introduction}

Consider a pure $d$-dimensional polyhedral complex $\Delta$.
The {\em graph $\G(\Delta)$}, or {\em $1$-skeleton}, of $\Delta$ is the undirected simple graph that has the 
vertices of $\Delta$ as nodes and the one-dimensional faces of $\Delta$ as edges.

The study of graph-theoretic  connectivity of skeleta of 
polyhedral  complexes has its beginning with Steinitz' Theorem from 1922, 
which states that a graph is the $1$-skeleton of the boundary
complex of
some $3$-dimensional convex polytope if and only if it is $3$-connected and planar. 
Later, Balinski \cite{Balinski1961} generalized 
part of this to higher dimensions by showing that the graph of every $(d+1)$-dimensional 
convex polytope is $(d+1)$-connected.
This was generalized further by Barnette \cite{Barnette1982} who showed that the 
graph of every $d$-dimensional polyhedral 
pseudomanifold is $(d+1)$-connected. 

In this paper we consider the simplicial case. 
A simplicial pseudomanifold is said to be {\em flag} if its faces coincide with the cliques
of its $1$-skeleton, and it is said to be {\em normal} if
all links of faces are connected.
It was shown by Athanasiadis \cite{Athanasiadis2011} that the graph of a flag
simplicial $d$-pseudomanifold is $2d$-connected.

We introduce an invariant $b_{\Delta}$ of a pure $d$-dimensional simplicial complex
$\Delta$ taking values in the range $0\le b_{\Delta} \le d-1$. 
It is shown to affect connectivity in the following way.

\begin{theorem} \label{main}
Let ${\Delta}$ be a $d$-dimensional normal simplicial pseudomanifold. 
Then the graph $\G(\Delta)$ is $(2d-b_{\Delta})$-connected.
\end{theorem}

The definition of $b_{\Delta}$ involves the
concept of {\em banner}  triangulations, a generalization of
flag triangulations. We have that $b_{\Delta} =0$ if and only if the complex ${\Delta}$ is banner,
and in particular if $\Delta$ is flag. Thus, for the case of normal simplicial pseudomanifolds
our result interpolates between the theorems of Barnette
and Athanasiadis. 

\section{Preliminaries}

In this section we collect some definitions and review some auxiliary results needed later in the paper.

\subsection{Simplicial complexes}

We assume basic knowledge about simplicial complexes.
Throughout the paper, $\Delta$ will denote a pure finite
$d$-dimensional simplicial complex on vertex set $V$.

Let $\tau$ be a face of $\Delta$. The {\em link} of $\tau$ in $\Delta$, 
denoted $\link_\Delta(\tau)$,
is the subcomplex 
that contains a face $\sigma \in \Delta$ if 
$\sigma\cap\tau=\emptyset$ and $\sigma\cup\tau \in \Delta$. 

Let $x \in V$ be a vertex of $\Delta$. 
The {\em closed star} $\star_\Delta(x)$ is the subcomplex of $\Delta$ which is the cone over $\link_\Delta(x)$ with apex $x$.
The 
 {\em antistar} $\astar_\Delta(x)$ is the subcomplex of $\Delta$ induced
 on the set of vertices $V\setminus \{x\}$. Note that
 $\star_\Delta(x) \cap \astar_\Delta(x) = \link_\Delta(x)$.

A pure $d$-dimensional  complex $\Delta$ is a {\em pseudomanifold} if 
\begin{itemize}
\item[(i)] every $(d-1)$-dimensional  face is contained in exactly two facets (maximal faces),
\item[(ii)] $\Delta$ is strongly connected, 
meaning that the {\em facet graph} (whose vertices are the facets and edges the pairs of adjacent facets) of $\Delta$ is connected.
\end{itemize}

We use the following property of pseudomanifolds at a crucial point in the paper. 

\begin{lemma}[{\cite[Lemma 2]{Barnette1982}}]
\label{lem:astar_sc}
The antistar of any vertex in a pseudomanifold is strongly connected.
\end{lemma}

A pure simplicial complex $\Delta$ is a {\em homology manifold} if $\Delta$ is connected and
$\link_\Delta(\tau)$ has the homology of a sphere of the appropriate dimension
for every nonempty face $\tau \in \Delta$.
The complex $\Delta$ is a {\em homology sphere} if it is a homology manifold and itself has the homology of a sphere of the same dimension.

A pseudomanifold is called {\em normal} if all links  $\link_\Delta(\tau)$ of dimension at least one are connected.
The condition of being normal is quite
natural and holds e.g.\ for all homology manifolds.
Most importantly, it is not hard to check that the class of 
normal pseudomanifolds is closed under taking links. This is not the case for pseudomanifolds in general.

Of the following four properties for a pure simplicial complex $\Delta$,
each implies its successor:
\begin{itemize} 
\item[(i)] $\Delta$ is a triangulation of a compact topological manifold,
\item[(ii)] $\Delta$ is a homology manifold,
\item[(iii)] $\Delta$ is a normal pseudomanifold,
\item[(iv)]  $\Delta$ is a pseudomanifold.
\end{itemize}

\subsection{Graph theory}

We assume basic knowledge about graphs and refer to \cite{Diestel2005} for details.

A graph $G$ is said to be {\em $k$-connected} if $G$ has more than $k$ vertices and $G \setminus S$ is connected for every set of vertices $S$ with $|S| < k$, where $G \setminus S$ denotes the graph that one obtains by deleting the vertices in $S$ and all incident edges.

The following well-known theorem relates the connectivity of a graph to families of independent paths. Here a family of paths between two vertices $x$ and $y$ is called {\em independent} if the only vertices contained in more than one path of the family are $x$ and $y$.

\begin{theorem}[Menger's Theorem \cite{Diestel2005}]
A graph $G$ with at least $k+1$ vertices is $k$-connected if and only if any two vertices
can be joined by $k$ independent paths.
\end{theorem}

It turns out that it suffices to check this condition for vertices at distance two in the graph.
This fact plays a central role for proving the connectivity results in this paper.

\begin{lemma}[Liu's criterion, \cite{Liu1990}]
\label{liu}
Let $G$ be a connected graph with at least $k+1$ vertices. If for any two
vertices $u$ and $v$ of $G$ with distance $d_G(u, v) = 2$ there are $k$ 
independent $u$ --- $v$ paths in $G$, then $G$ is $k$-connected.
\end{lemma}

\begin{proof}
Assume that $G$ is not $k$-connected. Choose a set $S$ with less than $k$ vertices such that $G \setminus S$ is disconnected and such that every proper subset of $S$ does not disconnect $G$. Then there are two vertices $x,y$ in different components of $G \setminus S$ and a path from $x$ to $y$ in $G$ that contains exactly one element $s$ of $S$. Consider the vertices
$u$ and $v$  immediately before and after $s$ on this path. They are at distance two in $G$ and cannot be connected by a path in $G \setminus S$, because $x$ and $y$ cannot.
Thus, every $u$ --- $v$ path in $G$ must pass through $S$, so there are at most $|S| \leq k-1$ independent $u$ --- $v$ paths in $G$.
\end{proof}

\section{Banner complexes}

Here we present the concept of banner triangulations, a generalization
of flag triangulations.

\begin{definition}
\begin{enumerate} \item A {\em clique} is a subset $T\subseteq V$ such that
$\{u,v\}\in\Delta$ for all $u,v\in T$. It is a $j$-clique if $|T|=j$.
\item A clique $T$  is {\em spanning} if $T\in \Delta$.
\item The complex  $\Delta$ is said to be a {\em flag complex} if every clique is spanning.
\end{enumerate}
\end{definition}

Flag complexes have been shown to have very strong properties in many situations. This is also the case for the connectivity of their edge graphs $\G(\Delta)$ as the following result shows.

\begin{proposition}[Athanasiadis {\cite{Athanasiadis2011}}]
Let $\Delta$ be a  $d$-dimensional simplicial pseudomanifold. If $\Delta$ is flag,
 then the graph $\G(\Delta)$ is $2d$-connected. 
\end{proposition}

The aim of this paper is to interpolate between the connectivity result of Barnette 
for general pseudomanifolds
and the one of Athanasiadis for flag pseudomanifolds.
 For that purpose, we introduce the concept of banner complexes.

\begin{definition}
\begin{enumerate}
\item A clique $T$  is {\em critical} if  
$T\setminus \{v\}\in \Delta$, for some $v\in T$.
\item A pure $d$-dimensional complex $\Delta$ is said to be a {\em banner complex} if 
\begin{itemize}
\item every critical  $(d+1)$-clique is spanning and 
\item $\Delta$ does not contain the boundary complex of a $(d+1)$-simplex as a subcomplex.
\end{itemize}
\end{enumerate}
\end{definition}

\vspace{5mm}

\begin{figure}[h]
\begin{center}
\psfrag{A}{$A$}
\psfrag{B}{$B$}
\includegraphics{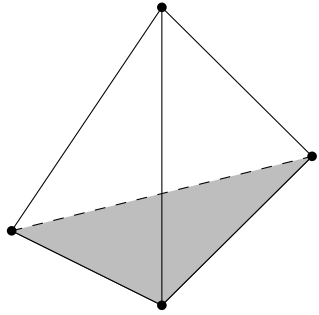}
\caption{A critical non-spanning $4$-clique}
\label{fig:banner1}
\end{center}
\end{figure}

\vspace{5mm}

Every clique in a flag complex is spanning.
Let us call
a pure $d$-dimensional complex $\Delta$, not containing the boundary of 
a $(d+1)$-simplex as a subcomplex, strongly banner
if every $(d+1)$-clique is spanning.
Then
$$\mbox{flag \ $\Rightarrow$ \ strongly banner \ $\Rightarrow$ \ banner}$$

All triangle-free graphs with no isolated vertices are strongly banner.
For two-dimensional complexes the concepts of banner and flag coincide. 
This is so because every $3$-clique is critical and a non-spanning $3$-clique
is the same thing as an empty triangle. Furthermore, if every $3$-clique is spanning but there is a non-spanning $4$-clique then the complex contains the boundary of a $3$-simplex as subcomplex.
The three concepts become distinct starting in dimension three,
as shown by the following examples.

\begin{example}
\label{rem0} 

For any graph $G=(V, E)$ we construct a pure  $3$-dimensional simplicial
complex $\Gamma(G)$ as follows. Just expand each edge $e$ of $G$ to a tetrahedron
$\sigma_e =\{v_1, v_2, e_1, e_2\}$, where $v_1$ and $v_2$ are the two endpoints
of $e$, and $e_1$ and $e_2$ are new vertices specific to $e$.
Then $\Gamma(G)$ is defined as the complex with facets 
$\sigma_e$, $e\in E$. Thus, $\Gamma(G)$ consists of tetrahedra, one for each edge $e\in E$,
that pairwise meet in a vertex exactly when the corresponding edges do.

Let $G=K_n$, the complete graph on $n$ vertices. One sees that
$\Gamma(K_3)$ is strongly banner but not flag, while 
$\Gamma(K_4)$ is  banner but not strongly banner.
\end{example}


\begin{example}
\label{rem1} 
We now present a detailed 
construction of a shellable  $3$-ball $\Gamma$ on $16$ vertices that is strongly banner but not flag.
The vertices are ($ i=1,2,3$):

\vspace{3mm}

$x_i, a_i, b_i, c_i, d_i$ and $y$, 
\vspace{3mm}

\noindent
and the 26 facets are ($i=1,2,3$):
\vspace{3mm}

$
\begin{array}{lllllll}
x_i, x_{i+1}, a_i, b_i &  && 
x_i, a_i, b_i, b_{i-1} & &&
y, a_1, a_2, a_3 \\
x_i, x_{i+1}, b_i, c_i & &&
x_i, a_i, a_{i-1}, b_{i-1} & &&
y, b_1, b_2, b_3 \\
x_i, x_{i+1}, c_i, d_i & &&
y, a_i, b_i, a_{i+1} & && \\
x_i, x_{i+1}, a_i, d_i & &&
y, b_i, a_{i+1}, b_{i+1}& &&
\end{array}
$

\vspace{3mm}

\noindent
Figure 2 shows part of the structure of $\Gamma$, namely 
the subcomplex generated by the $12$ facets

$
\begin{array}{llllllllll}
x_i, x_{i+1}, a_i, b_i &  && 
x_i, x_{i+1}, b_i, c_i &&&
x_i, x_{i+1}, c_i, d_i &  && 
x_i, x_{i+1}, a_i, d_i
\end{array}
$

\vspace{3mm}

\noindent
This subcomplex, a ``ring of three bananas,''
consists of  $3$ octahedra on vertices 
$\{x_i, x_{i+1}, a_i, b_i, c_i, d_i\}$, 
$i\in \Z /3\Z$, glued together
at the common vertices $x_1, x_2, x_3$.
The rest of $\Gamma$ is used to ``fill the hole'' in this octahedral ring.
 
\vspace{5mm}

\begin{figure}[h]
\begin{center}
\psfrag{a1}{$a_1$}
\psfrag{a2}{$a_2$}
\psfrag{a3}{$a_3$}
\psfrag{b1}{$b_1$}
\psfrag{b2}{$b_2$}
\psfrag{b3}{$b_3$}
\psfrag{c1}{$c_1$}
\psfrag{c2}{$c_2$}
\psfrag{c3}{$c_3$}
\psfrag{d1}{$d_1$}
\psfrag{d2}{$d_2$}
\psfrag{d3}{$d_3$}
\psfrag{x1}{$x_1$}
\psfrag{x2}{$x_2$}
\psfrag{x3}{$x_3$}
\includegraphics{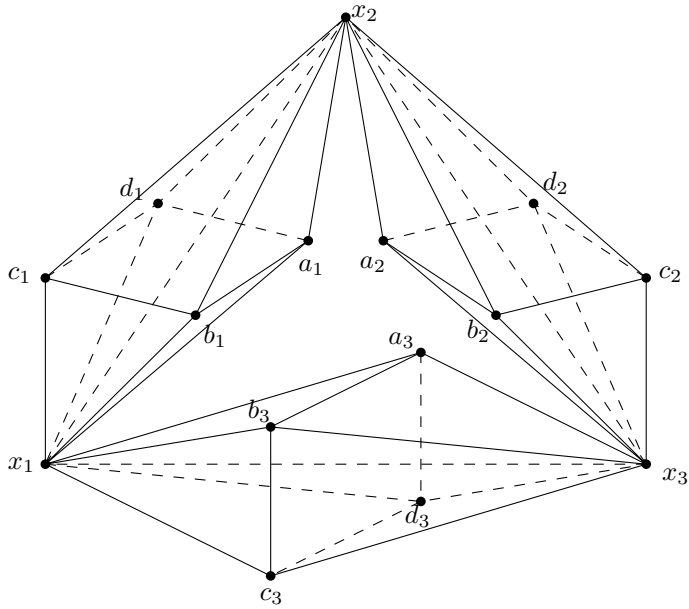}
\caption{The three bananas part of the ball $\Gamma$}
\label{fig:banner2}
\end{center}
\end{figure}

\vspace{5mm}

The complex $\Gamma$ is shellable. Shellings are obtained by starting
with the $8$ facets in $\star_{\Gamma}(y)$, at the center of $\Gamma$,
and then moving out to the facets of the three octahedra.
Thus, being a shellable pseudomanifold with boundary, $\Gamma$ is a 
$3$-dimensional ball. Its $f$-vector is
$(1, 16, 54, 65, 26)$.
Its boundary is a $2$-sphere with $f$-vector $(1, 15, 39, 26)$.

The complex $\Gamma$ has a unique empty triangle, namely $\{x_1, x_2, x_3\}$.
Hence, it is not flag. However, it is strongly banner. 
One can reason as follows to see that every $4$-clique is spanning.

Suppose that $F$ is a $4$-clique in $\G(\Gamma)$. Dividing the
vertices into $3$ groups

$$X=\{x_i\} \qquad  A=\{a_i, b_i, c_i, d_i\} \qquad Y=\{y\}
$$
we observe that none of the three sets contains a $4$-clique
and that  either $F\cap X=\emptyset$ or
$F\cap Y=\emptyset$, since there are no edges $\{x_i, y\}$.
So there are two cases to consider: 
$F\subseteq X\cup A$  and $F\subseteq  A\cup Y$.
We leave to the reader the few easy steps left to verify
that in both cases $F$ must be one of the $26$ facets.

From this ball a triangulated $3$-sphere that is strongly banner but not flag
can be derived, see Example \ref{rem2} below.
\end{example}

\vspace{2mm}

\begin{example} \label{rem9}
A triangulation of a $3$-ball that is banner but not
strongly banner can be constructed along the same lines as our 3-ball in
Example 3.5, starting this time from a $4$-clique, embedding its $6$ edges into
octahedra to form a ``tetrahedron of six bananas'',
and then filling in the rest so that it is banner.
\end{example}

\vspace{4mm}

The properties of being banner or strongly banner are inherited by some related
complexes, such as links, cones and suspensions. Here are a few useful
such constructions.

\begin{proposition}
\label{lem:banner_link}
Let $\Delta$ be a pure $d$-dimensional simplicial complex. If $\Delta$ is banner and $x$ is a vertex of $\Delta$, then $\link_\Delta(x)$ is banner. The same is true for being strongly banner.
\end{proposition}

\begin{proof}
Because $\link_\Delta(x)$ is $(d-1)$-dimensional, we need to show that every critical $d$-clique is spanning and that $\link_\Delta(x)$ does not contain the boundary of the $d$-simplex.

Let $T$ be a critical $d$-clique in $\link_\Delta(x)$. Then $T \setminus v$ is a face of $\link_\Delta(x)$ for some vertex $v \in T$. If we set $T' = T \cup \{x\}$, then $T'$ is a clique and $T' \setminus v$ is a face of $\Delta$. So, $T'$ is a critical $(d+1)$-clique and thus spanning because $\Delta$ is banner. Equivalently, $T$ is a face of $\link_\Delta(x)$ and thus spanning.

Assume that $\link_\Delta(x)$ contains the boundary of the $d$-simplex as a subcomplex. Then the set $T$ of vertices of that simplex in $\link_\Delta(x)$ is a critical $(d+1)$-clique in $\Delta$. Because $\Delta$ is banner, $T$ is spanning and the subcomplex of $\Delta$ on vertices $T \cup \{x\}$ is the boundary of a $(d+1)$-simplex.
\end{proof}

\vspace{4mm}

If a triangulated $d$-ball $\Delta$ is
extended by raising a cone with apex $x$ over its boundary complex  $\partial \Delta$,
we obtain a triangulated $d$-sphere  $\widetilde{\Delta} = \Delta \cup  
(\partial \Delta \ast x)$, see Figure \ref{fig:banner3}.

\setlength{\unitlength}{1cm}
\begin{figure}[ht]
\begin{center}
\psfrag{x}{$x$}
\psfrag{y}{$y$}
\psfrag{z}{$z$}
\psfrag{u}{$u$}
\psfrag{w}{$w$}
\includegraphics[width=0.4\textwidth]{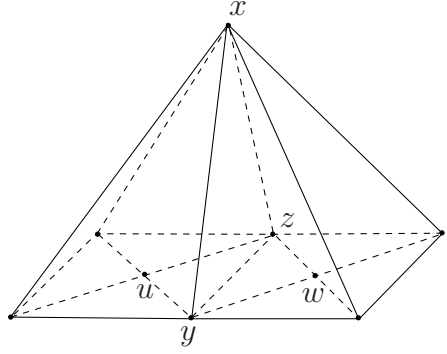}
\caption{The sphere $\widetilde{\Delta}$}
\label{fig:banner3}
\end{center}
\vspace{5mm}
\end{figure}

The $j$-cliques in $\G(\widetilde{\Delta})$ are of three kinds: 
\smallskip

(1) the $j$-cliques of  $\G(\Delta)$, 

(2) the $(j-1)$-cliques of $\G(\partial \Delta)$ augmented by $x$, and 

(3) the $(j-1)$-cliques of $\G(\Delta)$ with no interior vertex
and not of type (2),
\hspace*{10mm} augmented by $x$. 

For example, in Figure \ref{fig:banner3} the triangle $xyz$  is a $3$-clique of type $(3)$.

\begin{proposition}
\label{thm:conesusp} For a pure simplicial complex $\Delta$, 
let property X denote either ``banner'', ``strongly banner'', or ``flag''. Then,
\begin{enumerate}
 \item[(i)] 
${\Delta} \mbox{ has property X } \Leftrightarrow \; \mathrm{cone}(\Delta)  \mbox{ has property X}$ \\
 \item[(ii)]
${\Delta} \mbox{ has property X } \Leftrightarrow \; \mathrm{susp}(\Delta)   \mbox{ has property X}$ \\
 \item[(iii)] If $\Delta$ is a $d$-ball lacking $(j-1)$-cliques of type (3), then \\
$\widetilde{\Delta} \mbox{ has property X } \Leftrightarrow \; \Delta \mbox{ and } 
\partial \Delta\mbox{ have property X} 
$
\end{enumerate}
\end{proposition}

The proofs are in all cases straightforward verifications.


\begin{example}
\label{rem2} 
Parts (i) and (ii) of the proposition can be used to give examples in all dimensions of complexes that
are exactly one of  banner, strongly banner and flag.
Part (iii) applied to the $3$-ball $\Gamma$ of Example \ref{rem1} 
shows how to
construct a shellable $3$-sphere 
that is strongly banner but not flag.
Applying the construction instead to the $3$-ball 
of Example \ref{rem9} we obtain a banner shellable $3$-sphere that is not strongly banner.
\end{example}

\vspace{2mm}

The three classes of complexes that we have considered  have the property that
the graph $\G(\Delta)$ determines the whole complex, or almost. Any complex
is of course determined by its facets, and in a strongly banner
complex $\Delta$ the facets are determined by the $(d+1)$-cliques of $\G(\Delta)$.

The situation for banner complexes is slightly weaker. Suppose that
$\Delta$ is banner, strongly connected, and every codimension one face is contained in exactly
$k$ facets (e.g., if $\Delta$ is a banner pseudomanifold, the $k=2$ case). Then it can be shown that
if we know one particular $(d+1)$-clique to be a facet, then
we can identify the rest of the facets among the other $(d+1)$-cliques.



\section{Connectivity of banner pseudomanifolds}

The aim of this section is to prove the following:

\begin{theorem}
\label{thm:bannergraph}
Let $\Delta$ be a normal $d$-dimensional pseudomanifold.
If $\Delta$ is banner, then its  graph $\G(\Delta)$ is $2d$-connected.
\end{theorem}

To prove Theorem \ref{thm:bannergraph}, we will use a few lemmas.
For $x\in V$ let $N(x)$ be the set of vertices of $\star_{\Delta}(x)$,
that is, $x$ together with its neighboring vertices $z\in V$ such that $\{x,z\} \in \Delta$.

\begin{lemma}
\label{lem:banner_1}
Let $\Delta$ be a  banner pseudomanifold. If $\{x,y\}\in\Delta$
then $N(y) \not\subseteq N(x)$.
\end{lemma}

\begin{proof}
Let $F$ be any maximal face in $\link_\Delta(x)$ that contains $y$. Then,
since $\Delta$ is a pseudomanifold,  $\link_\Delta(F)$ consists of two isolated vertices, one of which is $x$. Let $w$ be the other such vertex, so that $F \cup \{w\} \in \Delta$ and $w \not\in F$. 

Assume that $w$ is adjacent to $x$, and as usual let $ d= \dim(\Delta)$.
For any vertex $v \in F$, the set $F \setminus v \cup \{x,w\}$ is a critical $(d+1)$-clique in $\Delta$, sincce 
$\left( F \setminus v \cup \{x,w\} \right) \setminus x = (F \setminus v) \cup w$
is a face of $\Delta$.
Because $\Delta$ is assumed to be banner, $F \setminus v \cup \{x,w\}$ is a face of $\Delta$.
Since $v \in F$ is arbitrary, this implies that $\Delta$ contains the boundary of the $(d+1)$-simplex on vertex set $F \cup \{x,w\}$, in contradiction to our assumption that it is banner. Therefore, $w$ is not adjacent to $x$ and 
$N(y) \not\subseteq N(x)$.
\end{proof}

\begin{lemma}
\label{lem:banner_Gnotcomplete}
Let $\Delta$ be a 
pseudomanifold. If $\Delta$ is banner, then $\G(\Delta)$ is not a complete graph.
\end{lemma}

\begin{proof}
Follows from the preceding Lemma and is also easy to see directly from the definition.
\end{proof}

We remark that banner complexes that are not pseudomanifolds can have a complete 
graph. However, this happens only for complexes that are the  
$d$-skeleton of  a $k$-simplex with $k \ge d+2$.
\medskip

Let $\Delta$ be a banner pseudomanifold, $x \in V(\Delta)$ an arbitrary vertex of $\Delta$ and
let $\Gamma$ denote the subcomplex of $\Delta$ 
induced on the set $V\setminus N(x)$. 
We need for our proof that  $\Gamma$ is nonempty and connected.
It is clear from Lemma \ref{lem:banner_1}  that $N(x)\neq V$.
That $\Gamma$ is connected  seems very natural
but is not to be taken 
automatically for granted. For instance, for the complex shown in 
Figure \ref{fig:banner3} (which is not banner)
the subcomplex  $\Gamma$ consists of two isolated vertices $u$ and $w$.

We offer two proofs. The first one is entirely elementary, relying on
Barnette's Lemma \ref{lem:astar_sc}. The second uses Lefschetz
duality, and is therefore valid only for homology manifolds. 

\begin{lemma}
\label{prop:banner_outsidepath}
Let $x$ be a vertex of  a banner pseudomanifold  $\Delta$. 
Then  the  subcomplex $\Gamma$ 
induced on the set of vertices not adjacent to $x$  is connected.
\end{lemma}


\begin{proof}[First proof.]
Let $u$ and $w$ be vertices not adjacent to $x$, and let
$\sigma$ and $\tau$ be facets of $\astar_{\Delta}(x)$
such that $u\in\sigma$ and $w\in\tau$.
From Lemma \ref{lem:astar_sc} we know that $\astar_\Delta(x)$ is strongly connected,
so we may choose a path $\sigma =\sigma_0 \rightarrow 
\sigma_1 \rightarrow \cdots\rightarrow  \sigma_k = \tau$
in the facet graph of $\astar_\Delta(x)$. 

This given, we want to  choose vertices $u_1, \ldots, u_k$ such that for all $i$:
\begin{enumerate}
\item $u_i\in \sigma_{i-1}\cap\sigma_i$, \item $u_i \notin \star(x) $
\end{enumerate}
If this is possible we are done, because then $ u=u_0 \rightarrow 
u_1 \rightarrow \cdots\rightarrow  u_k \rightarrow w$
is a path in $\astar(x) \setminus \star(x)$, that is, in $\Gamma$.

Suppose that for some $i$, such a choice of $u_i$ is impossible. Then
$$\sigma_{i-1}\cap \sigma_i \subseteq \star(x).
$$
The ridge $\mu :=\sigma_{i-1}\cap \sigma_i $ has the following properties:

1. $\mu\cup\{x\}$ is a $(d+1)$-clique,

2.  $\mu\cup\{x\}$ is critical, since $\mu\in\Delta$,

3. $\mu\cup\{x\}$ is not spanning, since $\Delta$ is a pseudomanifold and $\mu$ is already contained in the facets
$\sigma_{i-1} $ and $\sigma_i$.

This contradicts the assumption that $\Delta$ is banner.
%
%
%
%
\end{proof}

\begin{proof}[Second proof.]
Here we assume that $\Delta$ is a homology manifold. 
All homology groups are taken over $\Z_2$.
We have  that $H_d(\Delta) \cong \Z_2$, as is true for all pseudo-manifolds 
\cite[Exercise 43.5d]{Munkres1984}.

Let $\Sigma$ be the subcomplex of $\Delta$ induced on the 
set $N(x)$ and let  $\Gamma$ be the complex
induced on the complementary set of vertices,  as before. 
Consider the long exact sequence for relative homology 
\[
  H_d(\Sigma) \rightarrow H_d(\Delta) \rightarrow H_d(\Delta;\Sigma) \rightarrow H_{d-1}(\Sigma)
\]

 We have that $\star_{\Delta}(x)\subseteq \Sigma$, and our assumption that $\Delta$ is banner implies that $\star_{\Delta}(x)$ and $\Sigma$ 
have the same faces of dimensions $d$ and $d-1$. Because $\star_{\Delta}(x)$ is contractible, 
it follows that $H_d(\Sigma) = H_{d-1}(\Sigma) = 0$, and thus that $H_d(\Delta;\Sigma) \cong \Z_2$.

By Lefschetz duality \cite[Theorem 70.2]{Munkres1984}, we get that
$$H^0( ||\Delta|| \setminus ||\Sigma|| ) \cong H_d(\Delta;\Sigma) \cong \Z_2,$$
which means that the space $||\Delta|| \setminus ||\Sigma||$  is connected.
Finally, $\Gamma$ is a deformation retract of $||\Delta|| \setminus ||\Sigma||$  \cite[Lemma 70.1]{Munkres1984},
and is therefore also connected. This is equivalent to $\G(\Gamma)$ being connected.
\end{proof} 


We have now assembled all pieces needed to 
prove Theorem \ref{thm:bannergraph}.


\begin{proof}[Proof of Theorem \ref{thm:bannergraph}]
We use induction on $d$ and Liu's Lemma  \ref{liu}.
For $d=1$, $\Delta$ is a cycle graph and thus $2$-connected.

Assume that $d > 1$, and let $y$ and $z$ be a pair of vertices at distance two.
From Lemma \ref{lem:banner_Gnotcomplete} we know that $\G(\Delta)$ is not a complete graph, so
such pairs  exist. 

Let $x$ be a vertex of $\Delta$ such that $y$ and $z$ are contained in $\link_\Delta(x)$.
Because $\Delta$ is normal, $\link_\Delta(x)$ is a $(d-1)$-dimensional normal pseudomanifold. 
Using Proposition \ref{lem:banner_link} we see that $\link_\Delta(x)$ is banner. So, by induction, the  graph of $\link_\Delta(x)$ is $2(d-1)$-connected. This means that 
$\link_\Delta(x)$ contains more than $2(d-1)$ vertices, and
we can find $2(d-1)$ independent paths from $y$ to $z$ in $\link_\Delta(x)$. 

By Lemma  \ref{lem:banner_1} we may choose vertices $u\in N(y)\setminus N(x)$
and $w\in N(z)\setminus N(x)$. 
Thus, $u, w \in \Gamma$.
We obtain a path from $y$ to $z$ by first going to $u$. Then, by Lemma \ref{prop:banner_outsidepath},  we can continue along  a path from $u$ to $w$ with all vertices in $\Gamma$. Finally we go from $w$ to $z$. Note that this path has no interior vertices in $\link_\Delta(x)$.

Thus, together with the path $y \rightarrow  x \rightarrow z$ we have found $2d$ paths from $y$ to $z$ which by construction are independent. Also, $y$ and $z$ are not adjacent, so the path outside $\star_\Delta(x)$ has at least one inner vertex. This vertex and $x$ are both outside $\link_\Delta(x)$, so $\Delta$ has more than $2d$ vertices.
\end{proof}

\section{Complexes with banner links}

By definition,  a one-dimensional pseudomanifold is banner if and only if 
it is not the cycle graph $C_3$, the boundary of a triangle. 
However, every one-dimensional pseudomanifold -- including the boundary of a triangle -- is 
$2$-connected. This observation motivates including $C_3$ in the following definition.

\begin{definition}
Given a pure simplicial complex $\Delta$, let
\[
 b_{\Delta}= \min \{j  :  \link_\Delta(\sigma) \mbox{ is banner or $C_3$ for all } \sigma \in \Delta \mbox{ such
 that } |\sigma| = j \}.
\]
\end{definition}

We say that $b_{\Delta}$ is the {\em banner number} of $\Delta$.
Directly from the definition we get the following properties,
\begin{enumerate}
\item $0 \leq b_{\Delta} \leq d-1$,
\item $b_{\Delta} = 0 \Leftrightarrow \Delta$ is banner.
\end{enumerate}

\begin{lemma}
\label{lem:tentlinks}
Let $\Delta$ be a pure simplicial complex of dimension $d$ and let $\sigma \in \Delta$ be a face.
If $b_{\Delta} \geq |\sigma|$, then $b_{\link_\Delta(\sigma)} \leq b_{\Delta}- |\sigma|$.
\end{lemma}

\begin{proof}

It is a consequence of Proposition  \ref{lem:banner_link} that the link of any face $\tau$ with 
$|\tau | \ge b_{\Delta}$ is banner or $C_3$.
Let $\tau \in \link_\Delta(\sigma)$ be any face with $|\tau| \geq b_{\Delta} - |\sigma| \geq 0$. Then
$$\link_{\link_\Delta(\sigma)} (\tau) = \link_\Delta(\sigma \cup \tau)$$
is banner, because $|\sigma \cup \tau| = |\sigma| + |\tau| \geq b_{\Delta}$.
\end{proof}

We can now prove Theorem \ref{main}, stated in the introduction.

\begin{proof}[Proof of Theorem \ref{main}]
We use induction on $b_{\Delta}$. If $b_{\Delta} = 0$, then $\Delta$ is banner and the statement follows from Theorem \ref{thm:bannergraph}.

Assume that $b_{\Delta} > 0$ and let $x \in \Delta$ be any vertex. Let $\Gamma = \link_\Delta(x)$. By induction, $\G(\Gamma)$ is $(2(d-1)-b_{\Gamma})$-connected. Using Lemma \ref{lem:tentlinks}, we see that $2(d-1)-b_{\Gamma} \geq 2d - b_{\Delta} - 1$. Thus, $\G(\Gamma)$ is $(2d-b_{\Delta}-1)$-connected and has at least $2d - b_{\Delta}$ vertices. Together with $x$, then $\G(\Delta)$ has at least $2d-b_{\Delta}+1$ vertices.

If $\G(\Delta)$ is complete, then we are done. If not, let $z,y \in \Delta$ be two vertices at distance $2$ in $\G(\Delta)$. Then we find a vertex $x \in \Delta$ such that $y,z \in \link_\Delta(x)$ and again, let $\Gamma = \link_\Delta(x)$. By induction and Lemma \ref{lem:tentlinks}, there are $2d - b_{\Delta} - 1$ independent paths from $y$ to $z$ in $\G(\Gamma)$. Together with the path $y - x - z$, this gives us $2d - b_{\Delta}$ independent paths from $y$ to $z$ in $\G(\Delta)$. By Liu's Lemma \ref{liu}, this proves the result.
\end{proof}

\bigskip

\noindent
{\bf Acknowledgments.}
We want to thank C. Athanasiadis for bringing Barnette's Lemma \ref{lem:astar_sc}
to our attention, A. Goodarzi for suggesting Example \ref{rem0}, and S. Klee and I. Novik
for spotting an error in an earlier version. We also thank the anonymous referee who made made several 
suggestions that led to improvements of the exposition.

\end{document}